\documentclass{amsart}
\usepackage{amsmath,amsthm}
\usepackage{amsfonts,amssymb}
\usepackage{amsfonts}
\usepackage[all]{xy}
\usepackage{pdfsync}
\usepackage{hyperref}
\usepackage{graphicx}
\usepackage{subfigure}
\usepackage{color}
\usepackage{booktabs}
 \usepackage{multirow}
\allowdisplaybreaks[4]

\newtheorem{theorem}{Theorem}[section]
\newtheorem{lemma}[theorem]{Lemma}

\theoremstyle{definition}
\newtheorem{definition}[theorem]{Definition}

\theoremstyle{remark}
\newtheorem{remark}[theorem]{Remark}
\numberwithin{equation}{section}

\begin{document}

\title[Bounds for the $(\kappa, a)$-generalized Fourier kernel]{Bounds for  the kernel of the $(\kappa, a)$-generalized Fourier transform}


\author[H. De Bie]{Hendrik De Bie}
\address{Clifford research group \\Department of Electronics and Information Systems \\Faculty of Engineering and Architecture\\Ghent University\\Krijgslaan 281, 9000 Ghent\\ Belgium.}
\curraddr{}
\email{Hendrik.DeBie@UGent.be}
\thanks{}

\author[P. Lian]{Pan Lian}
\address{1. School of Mathematical Sciences\\ Tianjin Normal University\\
Binshui West Road 393, Tianjin 300387\\ P.R. China\\
2. Clifford research group \\Department of Electronics and Information Systems \\Faculty of Engineering and Architecture\\Ghent University\\Krijgslaan 281, 9000 Ghent\\ Belgium.}
\curraddr{}
\email{panlian@tjnu.edu.cn}
\thanks{}

\author[F. Maes]{Frederick Maes}
\address{Research Group NaM2\\Department of Electronics and Information Systems\\Faculty of Engineering and Architecture\\ Ghent University\\
Krijgslaan 281, 9000 Ghent\\ Belgium}
\curraddr{}
\email{Frederick.Maes@UGent.be}
\thanks{}

\date{}

\dedicatory{}


\begin{abstract} In this paper, we  study  the pointwise bounds for the kernel of the $(\kappa, a)$-generalized Fourier transform with  $\kappa\equiv0$,  introduced by  Ben Sa\"id,  Kobayashi and {\O}rsted. We present explicit formulas for the case $a=4$, which show that the kernels can exhibit polynomial growth. Subsequently, we provide a polynomial bound  for the even dimensional kernel for this transform, focusing on the cases  with finite order. Furthermore, by utilizing  an estimation  for the Prabhakar  function, it is found that the $(0,a)$-generalized Fourier kernel is bounded by a constant when $a>1$ and $m\ge 2$, except within an angular domain that diminishes as  $a \rightarrow \infty$. As a byproduct, we  prove that the $(0, 2^{\ell}/n)$-generalized Fourier kernel is uniformly bounded, when $m=2$ and $\ell, n\in \mathbb{N}$.
\end{abstract}

\maketitle
\section{Introduction}
The $(\kappa, a)$-generalized Fourier transform, denoted by $\mathcal{F}_{\kappa, a}$, is a two-parameter family of integral transforms. It was introduced  by Ben Sa\"id,  Kobayashi and {\O}rsted in \cite{bko1} and further investigated in detail in \cite{bko}. This transform can be considered as an `interpolation' between the Euclidean Fourier transform  and Hankel transform, with additional deformation from Dunkl operators \cite{dux}. In particular,  it  reduces to  the Dunkl transform \cite{dej},  when $a=2$.

One important question in the study of $\mathcal{F}_{\kappa, a}$  is to determine  the  boundedness of  its Schwartz distribution kernel, denoted by $B_{\kappa, a}(x, y)$. In the work \cite{git} of Gorbachev, Ivanov and Tikhonov, 
the following conjecture was proposed:
when $2\langle k\rangle+m+a\ge 3$, $B_{\kappa, a}$ is uniformly bounded by $1$, that is
    \begin{equation} \label{bb1}
        \left|B_{\kappa, a}(x,  y)\right|\le B_{\kappa, a}(0, y)=1, \qquad \qquad \forall\,  x, y\in \mathbb{R}^{m}.
    \end{equation}   
Here $\langle k\rangle$ is a constant arising  from the  Dunkl deformation. 
 One dimensional kernels are well understood so far, see 
 the recent paper \cite[Theorem 1.1]{gitt1}.  The bounds for higher dimensional kernels are much less understood, except  for the two significant cases, i.e. the Dunkl transform ($a=2$) and the Hankel transform  ($a=1$). For $a=2/n$ with $n\in \mathbb{N}$ and $\kappa\equiv 0$, inequality \eqref{bb1} was confirmed in \cite[Theorem 9]{cdl}. Recently,  the three authors mentioned above presented a negative result in \cite[Theorem 1.2]{gitt1} indicating that \eqref{bb1} is not valid for some parameter values. More precisely,  they prove that 
 \begin{equation*}
     \|B_{\kappa, a}(x,  y)\|_{\infty}>1, \qquad\qquad   x, y\in \mathbb{R}^{m}, 
 \end{equation*} 
 when $m\ge 2$, $a\in (1,2)\cup (2, \infty)$ and $\langle \kappa\rangle \ge 0$. This is achieved by examining the kernel's behavior when the product of $|x|$ and $|y|$ is sufficiently small. At those points, the behavior of the kernel is primarily determined by the first two terms in the  series expansion,  see the subsequent formula \eqref{ks} for $\kappa\equiv 0$. 
 Further information regarding the multivariate kernel's behavior remains unknown.
  
  The investigation  of the generalized Fourier kernel $B_{\kappa, a}(x, y)$ crucially involves the Dunkl's intertwining operator \cite{dux}, and the analysis of the Neumann-type series of Bessel functions of the first kind given in \cite[Eq. (4.49)]{bko}. The significant cases of the latter series are the generalized Fourier kernels arising only from radial deformation, i.e. $\kappa\equiv 0$. However, it is important to note that in general $\widetilde{V}_{\kappa}[B_{0,a}(\,\cdot\, , y)](x)\neq B_{\kappa, a}(x, y)$. Here $\widetilde{V}_{\kappa}$ is an induced operator by the Dunkl's intertwining operator \cite[Eq. (2.6)]{bko}.  Our specific focus  in this paper will revolve around the kernel of the radially deformed Fourier transform,  denoted by $\mathcal{F}_{a}$ for short.  This transform was formally defined as follows.
\begin{definition}
    Let $a\in \mathbb{R}^{+}$, the radially deformed Fourier transform $\mathcal{F}_{a}$ is defined as the unitary operator   
\begin{equation} \label{rf}
    \mathcal{F}_{a}=\exp\left[\frac{i\pi(m+a-2)}{2a}\right] \exp\left[\frac{i\pi}{2a}\left(|x|^{2-a}\Delta-|x|^{a}\right)\right]
\end{equation}
 on $L^{2}(\mathbb{R}^{m}, |x|^{a-2}\,{\rm d}x).$ 
\end{definition}
\begin{remark}
 (a) When $a=2$, \eqref{rf} reduces to the equivalent description of the Euclidean Fourier transform
discovered by Howe in \cite{how}, using the quantum harmonic oscillator $-(\Delta-|x|^{2})/2$.

 (b) When $a=1$, it corresponds to  the Hankel transform serving   as the unitary inversion operator of the Schr\"odinger model of the minimal representation of the group ${\rm O}(m+1, 2)$ in  \cite{km}.  
\end{remark}
 \begin{remark} $\mathcal{F}_{a}$ is of finite order if and only if $a\in \mathbb{Q}_{+}$. Suppose  $a\in \mathbb{Q}_{+}$ is of
the form $a=p/q$, with $p, q\in \mathbb{N}$, then $(\mathcal{F}_{a})^{2p}={\rm Id}$. In particular, $\mathcal{F}_{a}^{-1}=\mathcal{F}_{a}^{2p-1}$.  
\end{remark}

Its distribution kernel is given by a series in terms of Bessel functions and Gegenbauer polynomials, see \cite{bko, de1}. Explicitly, we have,
\begin{theorem} \label{kk1} For $x, y\in \mathbb{R}^{m}$ and $a>0$, the series
\begin{equation} \label{ks}
    K_{a}^{m}(x, y)=a^{2\lambda/a}\Gamma\left(\frac{2\lambda+a}{a}\right)\sum_{k=0}^{\infty}
    e^{-\frac{i\pi k}{a}} \frac{\lambda+k}{\lambda}z^{-\lambda}J_{\frac{2(k+\lambda)}{a}}\left(\frac{2}{a}z^{a/2}\right) 
    C_{k}^{(\lambda)}(\xi),
\end{equation}
     converges absolutely and uniformly on compact subsets, where $\lambda=(m-2)/2$, $z=|x||y|$ and $\xi=\langle x, y\rangle /z$.  The operator $\mathcal{F}_{a}$ defined in \eqref{rf} coincides with the integral transform
    \begin{equation*}
        \mathcal{F}_{a}[f](y)=\frac{\Gamma(m/2)}{\Gamma\left(\frac{2\lambda+a}{a}\right)2a^{2\lambda/a} \pi^{m/2}}
        \int_{\mathbb{R}^{m}} K_{a}^{m}(x, y) f(x)|x|^{a-2} \,{\rm d}x,
    \end{equation*}
   defined on a dense subset of  $L^{2}(\mathbb{R}^{m}, |x|^{a-2}\, {\rm d}x)$. 
\end{theorem}
\begin{remark}  When $a=1$, it is known that
    \begin{equation*} 
          K_{1}^{m}(x, y)=\Gamma\left(\frac{m-1}{2}\right)\widetilde{J}_{\frac{m-3}{2}}\left(\sqrt{2\left(|x||y|+\langle x, y\rangle\right)}\right),
    \end{equation*}
    with $\widetilde{J}_{\nu}(z)=(z/2)^{-\nu}J_{\nu}(z)$, see \cite{km}. While for $a=2$, we recover $K_{2}^{m}(x, y)=e^{-i\langle x, y\rangle}$,  see e.g.\,\cite{bko}.
\end{remark}

A general closed expression for the  kernel \eqref{ks} in terms of elementary functions is not available now. However, it has been found that  a closed expression exists in the Laplace domain  in \cite{cdl}. This was achieved  by introducing an auxiliary variable $t$ in the series \eqref{ks}, where the variable $2z^{a/2}/a$ of the Bessel function is replaced by $2 z^{a/2} t/a$ and then Laplace transformed  with respect to the new variable $t$,  see \cite[Eq.(8)]{cdl}.   
\begin{theorem} \label{lapp1} For $a>0$, $m\ge 2$ and ${\rm Re}\, s$ big enough, the  kernel of $\mathcal{F}_{a}$ in the Laplace domain is given by
\begin{equation}\label{lap1}
    \mathcal{L}[K_{a}^{m}(x,y, t)](s)=2^{2\lambda/a}\Gamma\left(\frac{2\lambda+a}{a}\right)\frac{1}{r}\left(\frac{1}{R}\right)^{2\lambda/a}\frac{1-u_{R}^{2}}{(1-2\xi u_{R}+u_{R}^{2})^{\lambda+1}},
\end{equation}    
where $u_{R}=\left(e^{\frac{-i\pi}{2}}z_{a}/R  \right)^{2/a}$ with $r=\sqrt{s^{2}+z_{a}^{2}}$,  $R=s+r$ and $z_{a}=\frac{2}{a}z^{a/2}$, $\lambda$, $z$ and $\xi$ are defined in Theorem \ref{kk1}.
\end{theorem}
\begin{remark} When $a=2/n$ with $n\in \mathbb{N}$, an explicit formula  for the  kernel was  obtained by  performing the inverse Laplace transform and subsequently setting the auxiliary variable $t$ to 1. Moreover, the optimal uniform bound is found to be 1 for both even and odd dimensions,  see \cite[Theorem 9]{cdl}.  An alternative proof  for the explicit expression was later given in \cite{dd}. It also closely related with the Dunkl kernel associated to the dihedral groups, see \cite{cdl}.
\end{remark}

In this paper, we further investigate the  behavior of these  generalized Fourier kernels. Observing the parameters of the Bessel functions in the kernel series \eqref{ks},   particular attention should be given to the case when $a=4$.
Indeed, our explicit expressions  for $a=4$ and $m>2$ in Section \ref{secmot}
show that the kernels are not uniformly  bounded, but instead exhibit a polynomial growth. This  contrasts  with the known results for $a=2/n$ and  differs with what many researchers have expected previously.

Using its Laplace domain expression \eqref{lap1}, we provide a polynomial bound for the kernel of the radially deformed Fourier transform of finite order, i.e. $a=p/q$, with $p, q\in \mathbb{N}$ and $m=2n$, see  Theorem \ref{pes} below. This allows us to introduce the function space on which $\mathcal{F}_{a}$ is well defined. The two dimensional kernels exhibit some differences, as the reproducing kernels of spherical harmonics reduce to $\cos k\theta$. Studies for the generalized Fourier transforms with polynomial bounded kernels can be found in e.g. \cite{dov, dx, GJ}.

An integral expression and a bound are given for the Prabhakar generalized Mittag-Leffer function in Section \ref{sec:prab}. Based on the obtained estimate,  we show that the $(0,a)$-generalized Fourier kernel is bounded by a constant when $a>1$ and $m\ge 2$, except within an angular domain that diminishes as  $a \rightarrow \infty$.
This means that the generalized Fourier kernel \eqref{ks}  is uniformly bounded on an unbounded domain in $\mathbb{R}^{m}\times \mathbb{R}^{m}$, but it may exhibit polynomial growth on the remaining region. For $K_{4}^{2n}(x, y)$,  this can  be seen  from the  asymptotic expansion given in Remark \ref{rem1}.
As a byproduct, we  prove that the $(0, 2^{\ell}/n)$-generalized Fourier kernel is uniformly bounded by constants, when $m=2$ and $\ell, n\in \mathbb{N}$ in Section \ref{ufb2}.

For the readers' convenience, we collect the bounds  for the radially deformed Fourier kernel obtained in this paper in Table \ref{tab:example}.
\begin{table}
        \centering
         \caption{Bounds for the  kernel $K_{a}^{m}(x,y)$,  ($n\in \mathbb{N},\ell\in \mathbb{N}_{0}$)}
         \label{tab:example}
\begin{tabular}{cccc}
   \toprule
   $m$ (dim)   & $a$  & Theorem & bound   \\
   \midrule
    $2$ & $2^{\ell}/n$ & \ref{kj1} & $C$\\
   $2n$ & $4$ &\ref{kj2} & $C(1+|x||y|)^{\frac{m-2}{2}}$  \\
   $2n$ & $p/q\in \mathbb{Q}_{+}$ &\ref{pes} & $C(1+|x||y|)^{\frac{3mp}{2}-m+2}$  \\
   \bottomrule
\end{tabular}
\end{table}

 The remainder of this paper is organized as follows. In Section \ref{secmot}, we calculate the kernel with $a=4$, which suggests the polynomial bounds for general cases. Section \ref{sec:boundsopq} is devoted to the polynomial bounds for the even dimensional $(0, p/q)$-generalized Fourier kernel. In Section \ref{sec:prab}, we present estimates based on the Prabhakar function. In the last section, we show that this estimate  can be used to obtain the uniformly boundedness of certain kernels.
\section{The motivating case $a=4$} \label{secmot}
In this section, we investigate  kernels  with parameter $a=4$. It shows that even in these  simple cases, the behavior of the  kernels $K_{4}^{2n}(x, y)$ differs significantly from the known cases.

We start with the kernel of dimension two, which means $\lambda=0$. Using the well-known relation \cite[Eq. (4.7.8)]{sz}
\begin{equation*}
    \lim_{\lambda\rightarrow 0}\lambda^{-1} C^{(\lambda)}_{k}(\xi)=(2/k)\cos k\theta, \qquad \xi=\cos\theta,\,  k\ge 1,
\end{equation*} 
the generalized  Fourier kernel in \eqref{ks} reduces to \begin{equation} \label{ker2}
\begin{split}
  K_{a}^{2}(z, \xi)=&\,\lim_{\lambda\rightarrow 0}K_{a}^{m}(z,\xi)\\
                   =&\,J_{0}(z_{a})+2\sum_{k=1}^{\infty}e^{-\frac{i\pi k}{a}}J_{\frac{2k}{a}}(z_{a})\cos k\theta,  
\end{split} 
\end{equation}
where $z_{a}=\frac{2}{a}z^{a/2}$ and $\xi=\langle x, y\rangle /z=\cos \theta$ (see also \cite{de1}). 
Here and in the sequel, we use $K_{a}^{m}(z,\xi)$ and $K_{a}^{m}(x,y)$ with abuse of notation, as the meaning should be clear from the context.

Let ${\rm erf} (w)$ be the error function defined by
\begin{equation}\label{erf}
  {\rm erf} (w)=\frac{2}{\sqrt{\pi}}\int_{0}^{w}e^{-t^{2}}\,{\rm d}t,  
\end{equation}
and ${\rm erfc}(w)= 1-{\rm erf}(w)$ be the complementary error function.  Note that 
\begin{equation*}
    \frac{{\rm d}^{n+1}}{{\rm d}w^{n+1}}{\rm erf} (w)=(-1)^{n}\frac{2}{\sqrt{\pi}}H_{n}(w)e^{-w^{2}}, \qquad n\in \mathbb{N}_{0},
\end{equation*}
where $H_{n}(w)$ is the ordinary Hermite polynomial, see \cite{olb}. This property will help to illustrate the polynomial growth of high dimensional kernels. 
\begin{theorem} The two dimensional radially deformed Fourier  kernel  in \eqref{ks} with parameter $a=4$  is given by 
\begin{equation} \label{k24}
    \begin{split}
         K_{4}^{2}(z, \xi)=&\, e^{-iz^{2}(\xi^{2}-1/2)} {\rm erfc}\left[-e^{-i\frac{\pi}{4}} z\xi \right]\\
    =&\, (1-i)\left(\frac{2}{\pi}\right)^{1/2}e^{-\frac{i}{2}z^{2}\cos 2\theta}\int_{-\infty}^{z\cos\theta} e^{it^{2}}\,{\rm d}t.  
    \end{split}
\end{equation}
Furthermore, it satisfies 
\begin{equation} \label{hb1}
    \left|K_{4}^{2}(z, \xi)\right|\le 1+2\sqrt{\frac{2}{\pi}},  \qquad \forall\, (z,\xi)\in \mathbb{R}^{+}\times [-1, 1].
\end{equation}
In particular, 
\begin{equation*}
    \lim_{z\rightarrow +\infty} \left|K_{4}^{2}(z, 1)\right|=2.
\end{equation*}
\end{theorem}
\begin{proof} (1) Recall the following generating function of modified Bessel functions (see \cite[\S 3.3.1]{mos}),
\begin{equation} \label{gmb}
e^{w\cos\theta}{\rm erfc}\left[(2w)^{\frac{1}{2}}\cos\left(\frac{1}{2}\theta \right) \right]=I_{0}(w)+2\sum_{k=1}^{\infty}
    (-1)^{k}I_{\frac{k}{2}}(w)
\cos\left(\frac{k}{2}\theta\right),
\end{equation}
where $I_{\alpha}(z)$ is the modified Bessel function satisfying $I_{\alpha}(z)=i^{-\alpha}J_{\alpha}(iz)$ when the principal value of the phase  $-\pi\le \arg z\le \pi/2$ 
. Replacing $w=-iz_{4}=-iz^{2}/2$ and $\theta$ by $2\theta$ in \eqref{gmb}, it yields
\begin{equation} \label{gmb1} 
e^{-\frac{i}{2}z^{2}\cos 2\theta}{\rm erfc}\left[\frac{1-i}{\sqrt{2}} z\cos \theta\right]=J_{0}(z_{4})+2\sum_{k=1}^{\infty}
    (-1)^{k}e^{-\frac{i\pi k}{4}}J_{\frac{k}{2}}(z_{4})
\cos\left(k\theta\right).
\end{equation}
Substituting $\cos \theta$ by $-\cos\theta$ in \eqref{gmb1}, 
 a new expression for the kernel \eqref{ker2} follows, 
\begin{equation*}
    \begin{split}
        K_{4}^{2}(z, \xi)=&\, J_{0}(z_{4})+2\sum_{k=1}^{\infty}e^{-\frac{i\pi k}{4}}J_{\frac{k}{2}}(z_{4})\cos k\theta\\
        =&\, e^{-\frac{i}{2}z^{2}\cos 2\theta} {\rm erfc}\left[\frac{i-1}{\sqrt{2}} z\cos \theta \right].
    \end{split}
\end{equation*}

It is easy to see that the error function is an odd function from its definition \eqref{erf}, i.e. ${\rm erf}(-w)=-{\rm erf}(w)$. Using this property, the kernel $ K_{4}^{2}(z, \xi)$ now can be written as 
\begin{equation*}
\begin{split}
  K_{4}^{2}(z, \xi)=&\, e^{-iz^{2}(\xi^{2}-1/2)}+e^{-iz^{2}(\xi^{2}-1/2)}{\rm erf}\left[\frac{1-i}{\sqrt{2}} z \xi \right]\\
  =& \,e^{-iz^{2}(\xi^{2}-1/2)}+ \left(\frac{2}{\pi}\right)^{\frac{1}{2}}e^{-iz^{2}(\xi^{2}-1/2)}(1-i)\left[C(z \xi)+iS(z\xi)\right]\\
 =& \, e^{-iz^{2}(\xi^{2}-1/2)}+ \left(\frac{2}{\pi}\right)^{\frac{1}{2}}e^{-iz^{2}(\xi^{2}-1/2)}(1-i)\int_{0}^{z\xi} e^{it^{2}}\,{\rm d}t\\
 =& \,(1-i)\left(\frac{2}{\pi}\right)^{\frac{1}{2}}e^{-iz^{2}(\xi^{2}-1/2)}\int_{-\infty}^{z\xi} e^{it^{2}}\,{\rm d}t.
\end{split}   
\end{equation*}
Here $S(u)$ and $C(u)$  are the Fresnel  integrals defined by (see e.g. \cite[\S 9.2.4]{mos})
  \begin{equation*}
      S(u)=\int_{0}^{u}\sin (t^{2})\,{\rm d}t, \qquad C(u)=\int_{0}^{u}\cos(t^{2})\,{\rm d}t.
  \end{equation*}
In the second step, we have used the  relation 
\begin{equation*}
    C(u)+iS(u)=\sqrt{\frac{\pi}{2}}\cdot \frac{1+i}{2} {\rm erf}\left(\frac{1-i}{\sqrt{2}}u\right)
\end{equation*}
 and in the last step
\begin{equation}\label{frei}
    \lim_{u\rightarrow +\infty}S(u)=  \lim_{u\rightarrow +\infty}C(u)=\frac{1}{2}\cdot \sqrt{\frac{\pi}{2}}.
\end{equation}


(2) The bound of the kernel in \eqref{hb1} follows  from the obtained expressions, using \eqref{frei} and the fact that the Fresnel integrals $S(x)$ and $C(x)$ are bounded by $1$. Thus here we only consider the last limit,
\begin{equation*}
\begin{split}
    \lim_{z\rightarrow \infty} \left|K_{4}^{2}(z, 1)\right|=&\,\lim_{z\rightarrow\infty} \left|
     (1-i)\left(\frac{2}{\pi}\right)^{\frac{1}{2}}e^{-iz^{2}(1-1/2)}\int_{-\infty}^{z} e^{it^{2}}\,{\rm d}t
      \right| \\
      =&\,\frac{2}{\sqrt{\pi}}\cdot 2\left|\int_{0}^{\infty}e^{it^{2}}\,{\rm d}t \right|\\
      =&\, \frac{2}{\sqrt{\pi}}\cdot 2\left|\lim_{u\rightarrow +\infty}\left(C(u)+iS(u)\right)\right|\\
      =& \,2.
\end{split}     
\end{equation*}
Here we have used the property \eqref{frei} again.
\end{proof}
\begin{remark} Unlike for the Euclidean Fourier kernel $e^{-i\langle x, y\rangle}$, it holds that 
\begin{equation*}
    \left\|K_{4}^{2}(x,y)\right\|_{\infty}> \left|K_{4}^{2}(0, y)\right|=1, \qquad x,y\in \mathbb{R}^{2}.
\end{equation*}
\end{remark}
\begin{remark} In Cartesian coordinates, the expression for the kernel $K_{4}^{2}$ is 
    \begin{equation*} 
         K_{4}^{2}(x, y)=(1-i)\left(\frac{2}{\pi}\right)^{1/2}e^{-i\left(\langle x, y\rangle^{2}-\frac{|x|^{2}|y|^{2}}{2} \right)}\int_{-\infty}^{\langle x, y\rangle} e^{it^{2}}\,{\rm d}t.  
\end{equation*}
Direct calculations show that 
\begin{equation*} 
 \left\{
				\begin{array}{ll}
					|x|^{-2}\Delta_{x} K_{4}^{2}(x, y)=-|y|^{4} K_{4}^{2}(x, y),\\
					|y|^{-2}\Delta_{y} K_{4}^{2}(x, y)=-|x|^{4} K_{4}^{2}(x, y),\\
				\end{array}
				\right.
\end{equation*}
hold as expected.
\end{remark}

The  kernel of dimension four exhibits polynomial growth.
\begin{theorem} When $a=4$ and $m=4$, the kernel of $\mathcal{F}_{4}$ is given by
\begin{equation}\label{ker44}
\begin{split}
    K_{4}^{4}(z, \xi)=&\, e^{-iz^{2}(\xi^{2}-1/2)}\left(-2iz\xi \int_{-\infty}^{z\xi} e^{it^{2}}\,{\rm d}t+ 
    e^{iz^{2}\xi^{2}}\right) \\
    =&\,e^{iz^{2}/2}-2iz\cos\theta\, e^{-\frac{i}{2}z^{2}\cos 2\theta}\int_{-\infty}^{z\cos \theta} e^{it^{2}}\,{\rm d}t.
\end{split}
\end{equation} 
Furthermore, it is seen that
\begin{equation*}
   K_{4}^{4}(z, 1)=\mathcal{O}(z),\qquad z\rightarrow +\infty.
\end{equation*}
\end{theorem}
\begin{proof}
    The expression \eqref{ker44} follows from the recursive relation between $K_{a}^{m}$ and $K_{a}^{m+2}$ (see \cite[Lemma 1]{de1}). That is
    \begin{equation}\label{recs}
        K_{a}^{m+2}(z,\xi)=e^{i\frac{\pi}{a}}a^{\frac{2}{a}}\frac{\Gamma(\frac{2\lambda+a+2}{a})}{2(\lambda+1)\Gamma(\frac{2\lambda+a}{a})}z^{-1}\partial_{\xi}K_{a}^{m}(z, \xi),
    \end{equation}
    with $\lambda=(m-2)/2$,  and the fact 
   \begin{equation*}
       \partial_{\xi} \left(\int_{-\infty}^{z\xi} e^{it^{2}}\,{\rm d}t\right)=z e^{iz^{2}\xi^{2}}.
   \end{equation*}
 The  last limit relation follows from \eqref{ker44}. 
\end{proof}
 Using the recursive relation \eqref{recs} again, it yields that
    \begin{equation*}
         K_{4}^{6}(z, 1)=\mathcal{O}(z^{2}),\qquad z\rightarrow +\infty.
    \end{equation*}
 Moreover, 
 we have,  
\begin{theorem}\label{kj2} When the dimension $m=2n$ is even,  and $a=4$,  the kernel of $\mathcal{F}^{m}_{4}$ is  given by
\begin{equation*}
\begin{split}
    K_{4}^{m}(x,y)=&c_{n} 
   e^{\frac{i}{2} z^{2}\sin^{2}\theta}D_{-n}[(i-1)z\cos\theta],
\end{split}
\end{equation*}
where $c_{n}=2^{\frac{n}{2}}\Gamma\left((n+1)/2)\right)/\sqrt{\pi}$ and $D_{\nu}(z)$ is the parabolic cylinder function \cite[\S 12.1]{olb}. 
Furthermore, 
there exists a constant $C>0$ such that
    \begin{equation*}
        \left|K_{4}^{m}(x,y)\right|\le C\left(1+|x||y|\right)^{\frac{m-2}{2}}, \qquad x,y\in \mathbb{R}^{m}.
    \end{equation*}
    Here the growth order $(m-2)/2$ is optimal.
\end{theorem}
\begin{proof}
    The compact expression follows from the expression of  kernel $K_{4}^{2}$  given in \eqref{k24}, the recursive relation \eqref{recs} and the following formula \cite[\S 1.5.1 (17)]{bry}
    \begin{equation*}
        \frac{{\rm d}^{n}}{{\rm d} w^{n}}\left(e^{a^{2}w^{2}}{\rm erfc}(aw)\right)=\frac{2^{(n+1)/2}}{\sqrt{\pi}}n!(-a)^{n}
        e^{a^{2}w^{2}/2}D_{-n-1}(\sqrt{2}aw).
    \end{equation*}
 The optimal bounds can be seen from their expressions in terms of the Fresnel integrals.   
\end{proof}
\begin{remark} The notation $D_{\nu}(z)$ for parabolic cylinder functions is due to Whittaker. These functions are also known as the Weber parabolic cylinder functions $U(\mu, z)$, see \cite[\S 12.1]{olb}. They are related by  $D_{-\nu-\frac{1}{2}}(z)=U(\nu, z)$.     
\end{remark}
\begin{remark} \label{rem1}
The generalized Fourier kernel's asymptotic expansion for large variables can be obtained by setting $\nu=n-1/2$ in the following  (see \cite[\S 12.9]{olb}), which shows the polynomial growth of the kernel.   Let $\delta$ be an arbitrary small positive constant, we have as
$z\rightarrow \infty$,
\begin{align*}\begin{split}
  D_{-\nu-\frac{1}{2}}(z)=U(\nu, z)\sim &\, e^{-\frac{1}{4}z^{2}}z^{-\nu-\frac{1}{2}}\sum_{s=0}^{\infty}(-1)^{s}\frac{\left(\frac{1}{2}+\nu\right)_{2s}}{s!(2z^{2})^{s}}\\
   &\pm i\frac{\sqrt{2\pi}}{\Gamma\left(\frac{1}{2}+\nu\right)}
    e^{\mp i\pi \nu}e^{\frac{1}{4}z^{2}}z^{\nu-\frac{1}{2}}\sum_{s=0}^{\infty}\frac{\left(\frac{1}{2}-\nu\right)_{2s}}{s!(2z^{2})^{s}},  
\end{split}  
\end{align*}
when $\frac{1}{4}\pi+\delta\le \pm \arg z\le \frac{5}{4}\pi-\delta$ and
\begin{align*}\begin{split}
  D_{-\nu-\frac{1}{2}}(z)\sim & e^{-\frac{1}{4}z^{2}}z^{-\nu-\frac{1}{2}}\sum_{s=0}^{\infty}(-1)^{s}\frac{\left(\frac{1}{2}+\nu\right)_{2s}}{s!(2z^{2})^{s}},
\end{split}  
\end{align*}
 when $|\arg\, z|\le \frac{3}{4}\pi-\delta\left(<\frac{3}{4}\pi\right)$.   
\end{remark}

At the end of this section, we give an integral expression for $K_{6}^{2}(x, y)$. This method can also be applied to construct  other kernels of in  dimension 2 with even integer parameters. By combining with the subsequent Lemma \ref{fs1} and the recursive relation \eqref{recs}, it is in principle possible to give explicit expressions for all even dimensional kernels with rational parameters $a$.  Another way to derive these integral expressions is to use the Laplace domain expression \eqref{lap1}.

\begin{theorem} The two dimensional radially deformed Fourier kernel with $a=6$ is given by  
\begin{align*}
\begin{split}
  K_{6}^{2}(z, \cos\theta)=&\exp(-iz_{6}\cos3\theta)+2e^{-\frac{i\pi }{6}}J_{\frac{1}{3}}(z_{6})+2e^{-\frac{i\pi }{3}}J_{\frac{2}{3}}(z_{6})\\
    &+\sum_{k=1}^{2}e^{ik(\theta-\frac{\pi}{6})}\left(f_{1}\left(\frac{k}{3}, z_{6}, 3\theta-\frac{\pi}{2}\right)+if_{2}\left(\frac{k}{3}, z_{6}, 3\theta-\frac{\pi}{2}\right)\right)\\
    &+\sum_{k=1}^{2}e^{-ik(\theta+\frac{\pi}{6})}\left(f_{1}\left(\frac{k}{3}, z_{6}, 3\theta+\frac{\pi}{2}\right)-i f_{2}\left(\frac{k}{3}, z_{6}, 3\theta+\frac{\pi}{2}\right)\right),
\end{split}   
\end{align*}  
where $z_{6}=z^{3}/3$ and the functions $f_{1}$ and $f_{2}$ are defined by
\begin{align*}
    \begin{split}
       &f_{1}(\nu, z, \theta)=\frac{1}{4}{\rm cosec}\, \theta \int_{0}^{z}\sin[(z-t)\sin \theta]
       [\cos 2\theta J_{\nu}(t)+2\cos\theta J'_{\nu}(t)-J_{\nu+2}(t)]\,{\rm d}t,\\
       &f_{2}(\nu, z, \theta)=\frac{1}{2}\int_{0}^{z}\left(\frac{\nu}{t}+\cos \theta\right)\sin[(z-t)\sin \theta]J_{\nu}(t)\,{\rm d}t.
    \end{split}
\end{align*}
\end{theorem} 
\begin{proof} We rewrite the kernel series \eqref{ker2} as
\begin{eqnarray*}
   K_{6}^{2}(z, \xi)&=&J_{0}(z_{6})+2\sum_{k=1}^{\infty}e^{-\frac{i\pi k}{6}}J_{\frac{k}{3}}(z_{6})\cos k\theta\\
   &=&J_{0}(z_{6})+2\sum_{k=1}^{\infty}e^{-\frac{i\pi k}{2}}J_{k}(z_{6})\cos 3k\theta\\
    &&+2e^{-\frac{i\pi }{6}}\left[J_{\frac{1}{3}}(z_{6})+\sum_{k=1}^{\infty}e^{-\frac{i\pi k}{2}}J_{k+\frac{1}{3}}(z_{6})\cos (3k+1)\theta\right]\\
    &&+2e^{-\frac{i\pi }{3}}\left[J_{\frac{2}{3}}(z_{6})+\sum_{k=1}^{\infty}e^{-\frac{i\pi k}{2}}J_{k+\frac{2}{3}}(z_{6})\cos (3k+2)\theta\right]\\
   &=:&I+II+III.  
\end{eqnarray*}
For the first one $I=\exp(-iz_{6}\cos3\theta)$ we refer to \cite[\S 10.35.2]{olb}. In the following, we only compute $II$,
   \begin{equation*}
   \begin{split}
       2\sum_{k=1}^{\infty}e^{-\frac{i\pi k}{2}}J_{k+\frac{1}{3}}(z_{6})\cos (3k+1)\theta
       =&\, e^{i\theta} \sum_{k=1}^{\infty}J_{k+\frac{1}{3}}(z_{6})e^{3ik(\theta-\frac{\pi}{6})}\\
       &+e^{-i\theta} \sum_{k=1}^{\infty}J_{k+\frac{1}{3}}(z_{6})e^{-3ik(\theta+\frac{\pi}{6})}.
   \end{split}    
   \end{equation*} 
Now, using the formulas (see \cite[\S  5.7.10]{pbm2}) when ${\rm Re}\, \nu>0$, 
\begin{align*}
   \sum_{k=1}^{\infty}J_{k+\nu}(z)\sin k\theta=\frac{1}{2}\int_{0}^{z}\left(\frac{\nu}{t}+\cos \theta\right)\sin[(z-t)\sin \theta]J_{\nu}(t)\,{\rm d}t,  
\end{align*}
   and 
   \begin{align*}
   \begin{split}
       \sum_{k=1}^{\infty} J_{k+\nu}(z)\cos k\theta
       =&\,
       \frac{1}{4}{\rm cosec}\, \theta \int_{0}^{z}\sin[(z-t)\sin \theta]\\
       &\times [\cos 2\theta J_{\nu}(t)+2\cos\theta J'_{\nu}(t)-J_{\nu+2}(t)]\,{\rm d}t,
   \end{split}   
  \end{align*}
   we obtain the explicit expression. 
\end{proof}
\begin{remark} Using the relation $J_{\nu}'(x)=\frac{\nu}{x}J_{\nu}(x)-J_{\nu+1}(x)$, it is  seen that $f_{1}$ and $f_{2}$ 
can be bounded by a polynomial in $z$.  
This suggests a polynomial bound for the  kernels with general parameters.   
\end{remark}

 \section{Bounds for the $(0, p/q)$-generalized Fourier kernel}\label{sec:boundsopq}
       In this section, we provide a bound  for the even dimensional $(0, p/q)$-generalized Fourier kernel based on its  Laplace domain expression \eqref{lap1}. These transforms with rational parameters are of particular interest in harmonic analysis (see \cite{gitt1}),  as they are the only cases in this family  with finite order. Our proof  relies on the following estimate,  which is a direct generalization of Lemma 2 in \cite{cdl}. We omit its proof  for brevity.
        \begin{lemma} \label{phi2} Let $a_{j}\in \mathbb{R}$ with $j=1,\ldots, n$ and let $\alpha = (\alpha_{1}, \alpha_{2},\ldots,\alpha_{n}) \in \mathbb{N}^{n}$ be an index vector with length  $|\alpha| = \alpha_{1}+\alpha_{2} + \cdots + \alpha_{n}$.  Consider the function 
       \begin{equation*}
         F_{n, \alpha}(s)=\frac{1}{\prod_{j=1}^{n}(s+ia_{j})^{\alpha_{j}}},
         \end{equation*}
  whose inverse Laplace transform is denoted by
  \begin{equation*}
      f_{n,\alpha}(t)=\mathcal{L}^{-1}[F_{n,\alpha}(s)](t).
  \end{equation*}
  Then we have the following estimate
  \begin{equation*}
      |f_{n, \alpha}(t)|\le\frac{t^{|\alpha|-1}}{\Gamma(|\alpha|)}, \qquad \forall  t \in ]0, \infty[.
  \end{equation*}
\end{lemma}
\begin{remark}
    The function  $f_{n, \alpha}(t)$ can be represented  as a $\Phi_{2}$-Horn confluent hypergeometric function, see \cite{dd, pbm}.  Alternatively, the  uniform bound for $t\in [0, 1]$ can  be derived from its integral expression over a simplex.
\end{remark}

The main result of this section is the following. 
\begin{theorem} \label{pes}
    Let $a=p/q$ with $p,q\in \mathbb{N}$, and $m\ge 2$ be even. Then  there exists a constant $C>0$  such that
     \begin{equation*}
       \left|K_{p/q}^{m} (x,y)\right|\le C(1+|x||y|)^{\frac{3mp}{2}-m+2},
    \end{equation*}
for all $x,y\in \mathbb{R}^{m}$.
\end{theorem}
\begin{proof}  We will rewrite the compact formula \eqref{lap1}  for the kernel in the Laplace domain as a linear combination of several terms. Afterwards, we estimate each term in the time domain.

Recall that $\xi=\cos \theta$. Then the denominator  of  $\mathcal{L}[K_{a}^{m}(x,y, t)](s)$ in \eqref{lap1} can be factored as 
\begin{equation*}
    1-2\xi u_{R}+u_{R}^{2}=\left(u_{R}-e^{i\theta}\right)\left(u_{R}-e^{-i\theta}\right),
\end{equation*}
where  $u_{R}=\left(e^{\frac{-i\pi}{2}}z_{a}/R  \right)^{2/a}$ with $R=s+r$, $r=\sqrt{s^{2}+z_{a}^{2}}$ and $z_{a}=\frac{2}{a}z^{a/2}.$

Using the elementary identity 
\begin{equation*}
  x^{n}-a^{n}=(x-a)\left(x^{n-1}+ax^{n-2}+a^{2}x^{n-3}+\cdots+a^{n-1}\right),  
\end{equation*} it is seen that 
\begin{equation*}
    u_{R}-e^{i\theta}=\frac{u_{R}^{p}-e^{ip\theta}}{u_{R}^{p-1}+e^{i\theta}u_{R}^{p-2}+\cdots+e^{i\theta (p-1)}}.
\end{equation*}
It yields
\begin{equation}\label{ffy}
\begin{split} 
 1-2\xi u_{R}+u_{R}^{2}&=\frac{\left(u_{R}^{p}-e^{ip\theta}\right)\left(u_{R}^{p}-e^{-ip\theta}\right)}{\left(\sum_{k=0}^{p-1}u_{R}^{k}e^{i(p-1-k)\theta}\right)\left(\sum_{k=0}^{p-1}u_{R}^{k}e^{-i(p-1-k)\theta}\right)}\\
 &=\frac{\left(u_{R}^{p}-e^{ip\theta}\right)\left(u^{p}_{R}-e^{-ip\theta}\right)}
 {\sum\limits_{k=0}^{2(p-1)} \left(\sum\limits_{\substack{\ell+j=k\\0\le \ell, j\le p-1}} e^{i(\ell-j)\theta}\right) u_{R}^{k}  }.
\end{split} 
\end{equation}
Recalling that $z_{a}=\frac{2}{a}z^{a/2}$, the numerator in the right hand side of \eqref{ffy} can be rewritten as  
\begin{equation}\label{eq2}
\begin{split}
 \left(u_{R}^{p}-e^{ip\theta}\right)\left(u_{R}^{p}-e^{-ip\theta}\right)&= 1-2\cos(p\theta)u_{R}^{p}+u_{R}^{2p} \\  
 &= \frac{1}{R^{2q}}\left(R^{2q}-2(-1)^{q}\cos(p\theta)\left(z_{p/q}\right)^{q}+  
 \left(\frac{z_{p/q}^{2}}{R}\right)^{2q} \right)\\
 &= \frac{1}{R^{2q}}\left((r+s)^{2q}-2(-1)^{q}\cos(p\theta)\left(z_{p/q}\right)^{q}+  
 \left(r-s\right)^{2q} \right).
\end{split}
\end{equation}
A straightforward calculation shows that the term within  the bracket on  the right hand side of \eqref{eq2} is a polynomial in $s$ of degree $2q$. Furthermore, its $2q$ roots are symmetrically distributed on a circle with radius $z_{p/q}$, see \cite[Lemma 1]{cdl}. Explicitly, we have:
\begin{equation*}
 \left(u_{R}^{p}-e^{ip\theta}\right)\left(u_{R}^{p}-e^{-ip\theta}\right)
 =\frac{2^{2q}}{R^{2q}}\prod_{\ell=0}^{2q-1}\left(s+iz_{p/q}\cos\left(\frac{p\theta+2\pi\ell}{2q}\right) \right).  
\end{equation*}

Taking into account the aforementioned calculations,  the kernel $K_{p/q}^{m}(x, y, t)$ in the Laplace domain can now be expressed as follows:
\begin{equation}\label{q1}
\begin{split} 
\mathcal{L}\left[K_{p/q}^{m}(x,y, t)\right](s)\varpropto &\, \frac{1}{r}\left(\frac{1}{R}\right)^{\frac{2\lambda q}{p}}
    R^{2q(\lambda+1) }\left(\sum_{k=0}^{2(p-1)}\left(\sum\limits_{\substack{\ell+j=k\\0\le \ell, j\le p-1}} e^{i(\ell-j)\theta}\right)u_{R}^{k} \right)^{\lambda+1}
    \\ & \, \times \frac{(1-u_{R}^{2})}{\prod_{\ell=0}^{2q-1}\left(s+iz_{p/q}\cos\left(\frac{p\theta+2\pi\ell}{2q}\right) \right)^{\lambda+1}},     
\end{split}  
\end{equation}
where the notation $\varpropto$ means that the equality holds up to a constant factor. Note that when $m$ is even, $\lambda$ is an integer. Thus,  the right hand side of \eqref{q1} is a linear combination of terms of the form
\begin{equation} \label{pl1}
    \frac{1}{r}\left(\frac{1}{R}\right)^{\gamma} \frac{z_{p/q}^{\tilde{\gamma}}}{\prod_{\ell=0}^{2q-1}\left(s+iz_{p/q}\cos\left(\frac{p\theta+2\pi\ell}{2q}\right) \right)^{\lambda+1}},
\end{equation}
where $2\lambda q/p - 2q(\lambda +1) \leq \gamma \leq 2q(\lambda +1) - 2q\lambda/p$ and $0\le \tilde{\gamma} \le 4q(\lambda-\lambda /p+1).$ It is seen that those linear combination coefficients $c(\theta)$ are bounded by a constant.

In the remaining,  we only consider the most complicated term,  i.e. the term with the smallest value of $\gamma.$ Other terms can be estimated similarly. For this case, we expand the power of $R=r+s$ and subsequently split it into two parts based on the powers of $r$, distinguishing between even and odd powers,
\begin{equation}
\begin{split}
&\frac{1}{r}\left(\frac{1}{R}\right)^{\frac{2\lambda q}{p}}
     \frac{R^{2q(\lambda+1) }}{\prod_{\ell=0}^{2q-1}\left(s+iz_{p/q}\cos\left(\frac{p\theta+2\pi\ell}{2q}\right) \right)^{\lambda+1}}\\
     =&\, \frac{1}{r}\left(\frac{1}{R}\right)^{\frac{2\lambda q}{p}}
     \frac{\sum_{k=0}^{2q(\lambda+1)} \binom{2q(\lambda + 1)}{k} r^{k}s^{2q(\lambda+1)-k}}{\prod_{\ell=0}^{2q-1}\left(s+iz_{p/q}\cos\left(\frac{p\theta+2\pi\ell}{2q}\right) \right)^{\lambda+1}}\\
     =&: I+II
     \end{split}
\end{equation}
with 
\begin{equation*}
    \begin{split}
       I=\frac{1}{r}\left(\frac{1}{R}\right)^{\frac{2\lambda q}{p}} 
     \frac{\sum_{n=0}^{q(\lambda+1)} \binom{2q(\lambda+1)}{2n} r^{2n}s^{2q(\lambda+1)-2n} }{\prod_{\ell=0}^{2q-1}\left(s+iz_{p/q}\cos\left(\frac{p\theta+2\pi\ell}{2q}\right) \right)^{\lambda+1}}
    \end{split}
\end{equation*}
and
\begin{equation*}
     II=\left(\frac{1}{R}\right)^{\frac{2\lambda q}{p}} 
     \frac{\sum_{n=0}^{q(\lambda+1)-1}\binom{2q(\lambda+1)}{2n+1}r^{2n}s^{2q(\lambda+1)-2n-1} }{\prod_{\ell=0}^{2q-1}\left(s+iz_{p/q}\cos\left(\frac{p\theta+2\pi\ell}{2q}\right) \right)^{\lambda+1}}.
\end{equation*}
It is seen that both the numerators of the final factor in $I$ and $II$ are polynomials in $s$. Now, we rewrite the  rational factors on the right-hand side of $I$ and $II$ as follows,
\begin{eqnarray}\label{cer1}
     \frac{ r^{2n}s^{2q(\lambda+1)-2n} }{\prod_{\ell=0}^{2q-1}\left(s+iz_{p/q}\cos\left(\frac{p\theta+2\pi\ell}{2q}\right) \right)^{\lambda+1}}&=&\frac{ \left(s^{2}+z_{p/q}^{2}\right)^{n}s^{2q(\lambda+1)-2n} }{\prod_{\ell=0}^{2q-1}\left(s+iz_{p/q}\cos\left(\frac{p\theta+2\pi\ell}{2q}\right) \right)^{\lambda+1}}\nonumber\\
     &=&\frac{ \prod_{\ell=0}^{2q-1}\left[\left(s+iz_{p/q}\cos\left(\frac{p\theta+2\pi\ell}{2q}\right)\right)+c_{\ell} \right]^{\lambda+1}}{\prod_{\ell=0}^{2q-1}\left(s+iz_{p/q}\cos\left(\frac{p\theta+2\pi\ell}{2q}\right) \right)^{\lambda+1}},
\end{eqnarray}
where $0\le n\le q(\lambda+1)$ and each  $c_{\ell}$ satisfies \begin{equation*}
    |c_{\ell}|\le 2z_{p/q}, \quad \ell = 0,1,\dots, 2q-1.
\end{equation*}
Using the binomial theorem, it is seen that \eqref{cer1}  is indeed a finite linear combination of terms in the  form 
\begin{equation*} 
   \frac{d(z)}{\prod_{\ell=\ell_0}^{\ell_{1}}\left(s+iz_{p/q}\cos\left(\frac{p\theta+2\pi\ell}{2q}\right) \right)^{k_{\ell}}},
\end{equation*}
where $0\le \ell_{0}\le \ell_{1}\le 2q-1$ and  $0\le k_{\ell}\le \lambda+1$. The numerator $d(z)$ is a product of at most $2q(\lambda+1)$ factors $c_{\ell}$ (independent of $s$) and hence satisfies 
\begin{equation*}
    d(z)\le C\left(1+z_{p/q}^{2q(\lambda+1)}\right)\le C(1+z)^{p(\lambda+1)},
\end{equation*}
where $C$ is a constant. Note that this decomposition is not the usual partial fraction decomposition of a rational function. Estimating the constants that arise in the partial fraction decomposition would be inconvenient, as they may not be controlled by polynomials of $z=|x||y|$.  Collecting all we obtained,   $I$ is a finite linear combination of terms of the form
\begin{equation} \label{fr1}
  \frac{1}{r}\left(\frac{1}{R}\right)^{\frac{2\lambda q}{p}} \frac{d(z)}{\prod_{\ell=\ell_0}^{\ell_{1}}\left(s+iz_{p/q}\cos\left(\frac{p\theta+2\pi\ell}{2q}\right) \right)^{k_{\ell}}},
\end{equation}
while $II$ is a linear combination of terms of the form
\begin{equation} \label{fr2} 
 \left(\frac{1}{R}\right)^{\frac{2\lambda q}{p}} \frac{d(z)}{\prod_{\ell=\ell_0}^{\ell_{1}}\left(s+iz_{p/q}\cos\left(\frac{p\theta+2\pi\ell}{2q}\right) \right)^{k_{\ell}}}.
\end{equation}
Note that the factor $1/\left(\prod_{\ell=\ell_0}^{\ell_{1}}\left(s+iz_{p/q}\cos\left(\frac{p\theta+2\pi\ell}{2q}\right) \right)^{k_\ell}\right)$ has been studied in Lemma \ref{phi2}. 
\\

Now, we estimate the part of the kernel corresponding to each factor in \eqref{fr1} for $I$. Similar calculations  can be done for $II$. By the inverse Laplace transform formula (\cite[p.47 (15)]{pbm})
 \begin{equation*}
     \mathcal{L}^{-1}\left[\left(\frac{1}{R}\right)^{\nu}\right](t)=\frac{\nu}{t}\frac{J_{\nu}(at)}{a^{\nu}}, \qquad {\rm Re\,} \nu>0,  \,{\rm Re}\, s>|{\rm Im}\, a|, 
 \end{equation*}
and  (\cite[p.27 (3)]{pbm})
\begin{equation*}
    \mathcal{L}^{-1}\left[ \frac{1}{r}\right](t)=J_{0}(at), \qquad {\rm Re}\, s>|{\rm Im}\, a|,
\end{equation*}
with $r=\sqrt{s^{2}+a^{2}}$ and $R=s+r$, as well as  the convolution theorem for  Laplace transform, we have for $\lambda\in \mathbb{N}$ $   (m\neq 2),$
\begin{equation*}
    \mathcal{L}^{-1}\left[ \frac{1}{r}\left(\frac{1}{R}\right)^{\frac{2\lambda q}{p}}\right](t)= 
  \frac{2\lambda q}{p}\int_{0}^{t}J_{0}(au) \frac{J_{\frac{2\lambda q}{p}}(a(t-u))}{a^{\frac{2\lambda q}{p}}(t-u)}\,{\rm d}u. 
\end{equation*}
Using the convolution theorem again and evaluating at the point $t=1$, the corresponding part of the kernel $K_{p/q}^{m}$ in the time domain is given by
\begin{equation}\label{ff1}
\begin{split}
  &\mathcal{L}^{-1}\left[ \frac{1}{r}\left(\frac{1}{R}\right)^{\frac{2\lambda q}{p}} F_{n,\alpha}(s)\right](1)\\
   =&\,\frac{2\lambda q}{p}\int_{0}^{1} \int_{0}^{\tau}J_{0}\left(z_{p/q}\tau\right) \frac{J_{\frac{2\lambda q}{p}}\left(z_{p/q}(\tau-u)\right)}{z_{p/q}^{\frac{2\lambda q}{p}}(\tau-u)}{\rm d}u \,f_{n,\alpha}(1-\tau)\,{\rm d}\tau,   
\end{split} 
\end{equation}
where $F_{n, \alpha}$  is the function defined in Lemma \ref{phi2}.  Thus, the part of the kernel corresponding to  $I$ is indeed a finite linear combination of terms of the form 
\eqref{ff1} multiplied with some different $d(z)$.

We claim that for $\lambda \neq 0, $ i.e. $m>2$,   the right-hand side of \eqref{ff1} is bounded by a constant independent of $x$ and $y$. Indeed, it is smaller than
\begin{eqnarray*}
  &&C\int_{0}^{1} \int_{0}^{\tau}\left|J_{0}(z_{p/q}\tau) \frac{J_{\frac{2\lambda q}{p}}(z_{p/q}(\tau-u))}{z_{p/q}^{\frac{2\lambda q}{p}}(\tau-u)}\right|{\rm d}u \,|f_{n,\alpha}(1-\tau)|\,{\rm d}\tau \\
  &\le& C\int_{0}^{1} \int_{0}^{\tau}\frac{1}{(\tau-u)^{1-\frac{2\lambda q}{p}}}{\rm d}u\,|f_{n,\alpha}(1-\tau)|\,{\rm d}\tau\\
  &\le& C\int_{0}^{1} \int_{0}^{\tau}\frac{1}{(\tau-u)^{1-\frac{2\lambda q}{p}}}{\rm d}u\,{\rm d}\tau\\
 &=& C\int_{0}^{1} \tau^{1-\left(1-\frac{2\lambda q}{p}\right)}\int_{0}^{1}\frac{1}{(1-w)^{1-\frac{2\lambda q}{p}}} \,{\rm d}w\,{\rm d}\tau\\
    &\le&\, C,
\end{eqnarray*}
 where we  used  the following estimate for Bessel function of first kind in the first step, i.e.
 \begin{equation} \label{bs1}
  \left|z^{-\alpha}J_{\alpha}(z)\right|\le c, \qquad z\in \mathbb{R},
\end{equation}
where $c$ is a constant and $\alpha>-1/2$. Lemma \ref{phi2} is used in the second inequality. The third step is obtained by changing variables. The last step is due to $1-2\lambda q/p<1$.
 Inequality \eqref{bs1} follows   immediately from the following integral representation for the Bessel function,
\begin{equation*}
    J_{\alpha}(z)=\frac{\left(z/2\right)^{\alpha}}{\Gamma(\alpha+\frac{1}{2})\Gamma(\frac{1}{2})} 
    \int_{-1}^{1}e^{izu}(1-u^{2})^{\alpha-\frac{1}{2}} \,{\rm d}u, 
\end{equation*}
where $\alpha>-1/2$ and $z\in \mathbb{R}$, see \cite{olb}.

When $\lambda=0$, i.e.\ $m=2$, 
it is seen that the left hand side of \eqref{ff1} is bounded by a constant using the convolution theorem and Lemma \ref{phi2}. 

After counting  the powers of $z$, we complete the proof.
\end{proof}
\begin{remark} The polynomial growth order given here is not optimal. Since we are mainly interested when $z$ is large enough,  a smaller degree can be obtained if we use the property $|J_{\nu}(x)|\le 1$ for $\nu>0, x\in \mathbb{R}$ and the inverse Laplace transform formula \cite[p.48 (21)]{pbm}
\begin{equation*}
    \mathcal{L}^{-1}\left[\frac{1}{r}\frac{1}{R^{\nu}}\right](t)=a^{-\nu}J_{a\nu}(at),  \qquad {\rm Re}\, \nu>-1,
\end{equation*}
 with $r=\sqrt{s^{2}+a^{2}}$ and $R=r+s$.   
\end{remark}


With the  bounds established in  Theorem  \ref{lapp1}, we can now specify the domain in the definition
of  $\mathcal{F}_{p/q}$.  Let us proceed by defining  a class of functions as follows, 
\begin{equation*}
B_{p/q}(\mathbb{R}^{m}):=\left\{ f\in L^{1}(\mathbb{R}^{m}): \int_{\mathbb{R}^{m}}\left(1+|x|\right)^{\frac{3mp}{2}-m+2} |f(x)||x|^{p/q-2}{\rm d}x<\infty\right\}.
\end{equation*}
It immediately follows that,
\begin{theorem} Let $p, q\in \mathbb{N}$, and $m\ge 2$ be even. The $(0, p/q)$-generalized Fourier transform $\mathcal{F}_{p/q}$ is well-defined on $B_{p/q}(\mathbb{R}^{m})$. In particular, for $f\in B_{p/q}(\mathbb{R}^{m})$, $\mathcal{F}_{p/q}f$ is a continuous function.   
\end{theorem} 
\begin{proof}Since $3mp/2-m+2\ge 3>0$, we have \begin{equation*}
\begin{split}
   \left|K_{p/q}^{m} (x,y)\right|\le&\, C\left(1+|x||y|\right)^{\frac{3mp}{2}-m+2}\\
   \le&\, C(1+|x|)^{\frac{3mp}{2}-m+2}(1+|y|)^{\frac{3mp}{2}-m+2}. 
\end{split}    
    \end{equation*}  
 Thus,  $\mathcal{F}_{p/q}$ is well-defined on $B_{p/q}(\mathbb{R}^{m})$.  The continuity of $\mathcal{F}_{p/q}f$  follows from the continuity of the kernel and the dominated convergence theorem.  
\end{proof}

\begin{remark} With these bounds, several  uncertainty inequalities can now be established for $\mathcal{F}_{p/q}$   by  following the methods described  in \cite{GJ}. These include for instance the global uncertainty inequality,  Donoho-Stark’s uncertainty inequality, and Faris’s local uncertainty inequality.
\end{remark}

\section{Prabhakar function and kernel estimates}\label{sec:prab}
In this section, we begin by providing an integral expression and  an estimate for the Prabhakar generalized Mittag-Leffler function \cite{garg}. In case $\delta = 1$ we recover the estimate from \cite[Lemma 5]{fmkvb} which was derived using a similar technique. Subsequently, utilizing the obtained bounds, we establish an estimate for the 
$(0, a)$-generalized Fourier kernel in a certain domain. This  helps to locate the domain in which the kernel may grow rapidly. The results  are valid for all dimensions $m\ge 2$.

\begin{definition}The Prabhakar generalized Mittag-Leffler function is defined by
 \begin{equation*}
    E_{\alpha,\beta}^{\delta}(z):=\sum_{n=0}^{\infty}\frac{(\delta)_{n}}{n!\Gamma(\alpha n+\beta)}z^{n}, \qquad \alpha, \beta, \gamma \in \mathbb{C}, \quad  {\rm Re}\, \alpha >0,  
\end{equation*}
where $(\delta)_{n}=\delta(\delta+1)\ldots(\delta+n-1)$.
\end{definition}
\begin{remark} The  Laplace transform of the Prabhakar function is given by 
\begin{equation}\label{lapl1}
    \mathcal{L}\left(t^{\beta-1} E_{\alpha,\beta}^{\delta}(zt^{\alpha})\right)=\frac{1}{s^{\beta}}\frac{1}{(1-zs^{-\alpha})^{\delta}},
\end{equation}
where ${\rm Re} \,\alpha>0$, ${\rm Re} \,\beta>0$, ${\rm Re} \,s>0$ and $s>|z|^{1/({\rm Re} \,\alpha)}$, see \cite[(5.1.6)]{Gorenflo2014}.
\end{remark}
\begin{remark} When $\delta=1$, it reduces to the two-parametric Mittag-Leffler function \cite{Gorenflo2014}
\begin{equation*}
    E_{\alpha,\beta}(z):=\sum_{n=0}^{\infty} \frac{z^{n}}{\Gamma(\alpha n+\beta)}.
\end{equation*}
\end{remark}

For $\epsilon>0$ and $0 < \mu <\pi$,  we  consider the contour $\gamma(\epsilon, \mu)$ which is shown in Figure \ref{fig:contour}.
\begin{figure}[h!]
    \centering
    \includegraphics[scale=.5]{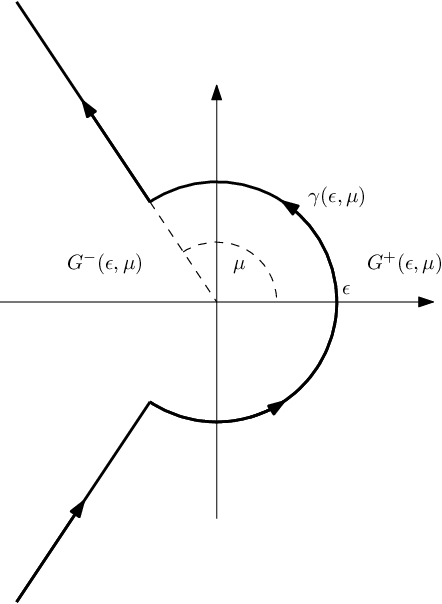}
    \caption{The contour $\gamma(\epsilon,\mu)$.}
    \label{fig:contour}
\end{figure}
It consists of the arc $\{ \zeta \in \mathbb{C} \mid |\zeta| = \epsilon, |\arg(\zeta)| \leq \mu\},$ and two rays $\{\zeta \in \mathbb{C} \mid |\zeta| \geq \epsilon, \arg(\zeta) = \pm \mu \}.$ The contour $\gamma(\epsilon,\mu)$ divides the complex plane in two parts, a region $G^-(\epsilon,\mu)$ located to the left of the contour and a region $G^+(\epsilon,\mu)$ located to the right of $\gamma(\epsilon,\mu).$ The orientation is such that $\arg$ is non-decreasing.

As was shown in \cite[Eq. (1.52)]{pod}, the reciprocal gamma function has the representation
\begin{equation}
    \label{eq:gammareciprocal}
    \frac{1}{\Gamma(z)} = \frac{1}{2\pi \alpha i} \int_{\gamma(\epsilon,\mu)} \exp\left(\zeta^{1/\alpha}\right) \zeta^{(1-z-\alpha)/\alpha} \, {\rm d}\zeta,
\end{equation}
for $\alpha < 2,$ and $\pi \alpha/2 < \mu < \min\{ \pi, \pi \alpha\}.$ The branch cut is taken to be the negative real axis for the function $z^\delta$ where $\delta >0$ is assumed. 

We give the following integral formula for the Prabhakar function.
\begin{theorem} \label{thm:Edabintegral}
    Let $\delta>0$, $0 < \alpha < 2$ and $\beta \in \mathbb{C}$ and let $\epsilon >0$ and $\mu$ be given such that
    $$
    \frac{\pi \alpha}{2} < \mu < \min\{\pi, \alpha \pi\}.
    $$
    Then for $z \in G^-(\epsilon,\mu)$ we have the formula
    $$
    E^\delta_{\alpha,\beta}(z) = \frac{1}{2 \pi \alpha i} \int_{\gamma(\epsilon,\mu)} \frac{ \exp\left(\zeta^{1/\alpha}\right) \zeta^{\frac{1-\beta}{\alpha} +\delta -1}}{ (\zeta-z)^\delta} \, {\rm d}\zeta
    $$
   and for $z \in G^+(\epsilon,\mu)$ in case $\delta$ is a positive integer $$ E^\delta_{\alpha,\beta}(z) = \frac{1}{2 \pi \alpha i} \int_{\gamma(\epsilon,\mu)} \frac{ \exp\left(\zeta^{1/\alpha}\right) \zeta^{\frac{1-\beta}{\alpha} +\delta -1}}{ (\zeta-z)^\delta} \, {\rm d}\zeta  + \frac{1}{\alpha \Gamma(\delta)}\frac{{\rm d}^{\delta-1}}{{\rm d}z^{\delta-1}} g(z),$$ where $g(z) = \exp(z^{1/\alpha}) z^{(1-\beta)/\alpha + \delta - 1}.$
\end{theorem}

\begin{proof}
The proof follows the same lines as for the two-parameter Mittag-Leffler function \cite[Section 1.2.7]{pod}.   Let us first consider the case where $|z| < \epsilon.$ Then it holds that $\left|z/\zeta\right| < 1$ for $\zeta \in \gamma(\epsilon, \mu).$ Next, we compute by means of \eqref{eq:gammareciprocal} and the binomial theorem
    \begin{align*}
    E^{\delta}_{\alpha,\beta}(z) &= \sum_{k = 0}^{+\infty} \frac{(\delta)_k}{k!} \frac{z^k}{2\pi \alpha i} \int_{\gamma(\epsilon,\mu)} \exp\left(\zeta^{1/\alpha}\right) \zeta^{(1-\alpha k - \beta - \alpha)/\alpha} \, {\rm d}\zeta\\
    & =\frac{1}{2\pi \alpha i} \int_{\gamma(\epsilon,\mu)} \exp\left(\zeta^{1/\alpha}\right) \zeta^{\frac{1-\beta}{\alpha} - 1} \sum_{k=0}^{+\infty} \frac{(\delta)_k}{k!} \left(\frac{z}{\zeta}\right)^k \, {\rm d}\zeta \\
    &= \frac{1}{2\pi \alpha i} \int_{\gamma(\epsilon,\mu)} \exp\left(\zeta^{1/\alpha}\right) \zeta^{\frac{1-\beta}{\alpha} - 1} \left( 1 - \frac{z}{\zeta}\right)^{-\delta}\,{\rm d}\zeta \\
    &= \frac{1}{2\pi \alpha i} \int_{\gamma(\epsilon,\mu)} \exp\left(\zeta^{1/\alpha}\right) \zeta^{\frac{1-\beta}{\alpha} + \delta -1} \left(\zeta - z\right)^{-\delta} \,{\rm d} \zeta.
    \end{align*}

    Now by the condition on $\mu$, we get that the integral is absolutely convergent and hence defines a function of the variable $z$, which is analytic in the region $G^-(\epsilon, \mu).$ Since the ball $|z|<\epsilon$ is contained in this region, by analytic continuation the above formula for $E^\delta_{\alpha,\beta}(z)$ holds for $z \in G^-(\epsilon,\mu).$

    In case $z \in G^{+}(\epsilon,\mu)$ and $\delta$ is assumed to be an integer, we pick $\eta > |z|$ so that $z \in G^-(\eta,\mu)$ and therefore $$ E_{\alpha,\beta}^{\delta}(z) = \frac{1}{2\pi \alpha i} \int_{\gamma(\eta, \mu)} \frac{\exp\left(\zeta^{1/\alpha} \right)\zeta^{ \frac{1-\beta}{\alpha} + \delta - 1}}{(\zeta-z)^\delta} \,{\rm d}\zeta. $$ The formula now follows by using Cauchy's integral formula for derivatives $$ \frac{1}{\alpha (\delta-1)!} \frac{{\rm d}^{\delta-1}}{{\rm d}z^{\delta-1}} g(z) = \frac{1}{2\alpha \pi i} \int_\chi \frac{1}{(\zeta-z)^\delta} \left(\exp\left(\zeta^{1/\alpha}\right) \zeta^{\frac{1-\beta}{\alpha} + \delta - 1}\right) \, {\rm d} \zeta, $$ where $\chi$ is the closed clockwise contour consisting of the arc of $\gamma(\eta,\mu)$, the arc of $\gamma(\epsilon,\mu)$  and the two line segments joining them.
\end{proof}

 
We now give an estimate for the function $E^\delta_{\alpha,\beta}(z)$ in some part of the complex plane.

\begin{theorem} Let $\delta >0$, $0<\alpha<2$ and $\beta>0$. Assume that $\alpha\pi/2<\mu<\min\{\pi, \alpha \pi\}$ and $\mu\le |\arg z|\le \pi$. Then there exists a constant $C>0$ only depending on $\mu, \alpha$ and $\beta$ such that
\begin{equation} \label{par}
    \left|E_{\alpha, \beta}^{\delta}(z)\right|\le \frac{C}{1+|z|^{\delta}}.
\end{equation}    
\end{theorem}

\begin{proof}
    Pick $\theta$ such that $\alpha\pi/2 < \theta < \mu$ and consider the representation from Theorem \ref{thm:Edabintegral},
    $$
    E^\delta_{\alpha,\beta}(z) = \frac{1}{2\pi \alpha i} \int_{\gamma(R, \theta)} \frac{\exp\left(\zeta^{1/\alpha}\right) \zeta^{\frac{1-\beta}{\alpha} + \delta -1}}{(\zeta-z)^\delta } \, {\rm d}\zeta,
    $$
    which is valid (by analytic continuation) for $z \in G^-(R, \theta).$ Here $R>0$ is  free to choose, but is from now on fixed. 

    In case $|z| >R $, we note that
    $$
    \min\limits_{\zeta \in \gamma(R,\theta)} \left|(\zeta -z)^\delta\right| \geq |z|^\delta \sin(\mu - \theta)^\delta
    $$
    and hence we estimate
    $$
    \left| E^\delta_{\alpha,\beta}(z)\right| \leq \frac{1}{2\alpha\pi |z|^\delta \sin(\mu-\theta)^\delta} \int_{\gamma(R,\theta)} \left|\exp\left(\zeta^{1/\alpha}\right)\right| \left| \zeta^{\frac{1-\beta}{\alpha} + \delta - 1}\right| \left|{\rm d} \zeta\right|.
    $$
    Denote with $I_\delta$ the latter integral. The contribution of the circular arc of $\gamma(R,\theta)$ to the integral $I_\delta$ is given by
    $$
    \int_{-\theta}^\theta \exp\left( R^{1/\alpha} \cos\left(\frac{u}{\alpha}\right)\right) R^{\frac{1-\beta}{\alpha} + \delta - 1}R\,{\rm d}u < \infty.
    $$
    The part of $I_\delta$ covering the two rays of $\gamma(R,\theta)$ is finite, as for $\zeta \in \gamma(R,\theta)$ with $\arg(\zeta) = \pm \theta$ and $|\zeta|\geq R$,  we have
    $$
    \left| \exp\left(\zeta^{1/\alpha}\right) \right|  = \exp\left(|\zeta|^{1/\alpha} \cos \frac{\theta}{\alpha}\right)
    $$
    and $\cos(\theta/\alpha)<0$ as $\pi/2 < \theta/\alpha < \pi$ by the choice of $\theta.$ Hence for $|z|>R$ and $\mu \leq |\arg(z)|\leq \pi$, we have
    $$
    \left|E^{\delta}_{\alpha,\beta}(z)\right| \leq \frac{C_1}{|z|^\delta},
    $$
    where $C_1$ is explicitly given by
    $$
    C_1 = \frac{1}{2\alpha\pi \sin(\mu - \theta)^\delta} \int_{\gamma(R,\theta)} \left|\exp\left(\zeta^{1/\alpha}\right)\right| \left| \zeta^{\frac{1-\beta}{\alpha} + \delta - 1}\right| \left|{\rm d} \zeta\right|.
    $$

    In the case $|z| \leq R,$ with $\mu \leq |\arg(z)| \leq \pi,$ we find
    \begin{equation*}
        \left| E^\delta_{\alpha,\beta}(z)\right| \leq \sum_{n=0}^\infty \frac{(\delta)_n}{n!} \frac{|z|^n}{\Gamma(\beta + \alpha n)}  \leq  E^\delta_{\alpha,\beta}(R). 
    \end{equation*}

    Summarizing, if we let $C = (1+R^\delta)\max\{ \frac{C_1}{R}, E^\delta_{\alpha,\beta}(R)\}$ then we obtain the desired  bound
    $$
    \left|E^\delta_{\alpha,\beta}(z)\right| \leq  \frac{C}{1+ |z|^\delta}, \qquad \mu \leq |\arg z| \leq \pi.
    $$
\end{proof}

\begin{remark} It is relatively easier to obtain a similar estimate when $\delta$ is an integer. These cases correspond to the even dimensional generalized Fourier kernels.   Indeed,  recall the  reduction formula in the third parameter for the Prabhakar function (see \cite[Eq.(2)]{garg})
\begin{equation*}
    E_{\alpha, \beta}^{\delta+1}(z)=\frac{E_{\alpha,\beta-1}^{\delta}(z)+(1-\beta+\alpha\delta)E_{\alpha,\beta}^{\delta}(z)}{\alpha\delta}.
\end{equation*}
 This means that it is possible to write the Prabhakar function, when $\delta \in \mathbb{N}$, in terms of the two parametric Mittag-Leffler function. Then a similar estimate (without $\delta$) follows from the known results below. 
\end{remark}
For the two-parametric Mittag-Leffler function,  we have (see \cite[p.35]{pod}),  
\begin{theorem} \label{ml2} If $\alpha<2$, $\beta$ is an arbitrary real number, $\mu$ is such that $\pi\alpha/2<\mu<\min\{\pi, \pi\alpha\}$ and $C$ is a real constant, then 
    \begin{equation*}
    |E_{\alpha,\beta}(z)|\le \frac{C}{1+|z|}, \qquad (\mu\le |\arg z|\le \pi,\quad 0\le|z|).
\end{equation*}
\end{theorem}


Now, we go back to the estimation of the radially deformed Fourier kernel. An integral representation  is obtained   by performing an inverse  Laplace transform of \eqref{lap1} in terms of the Prabhakar function using \eqref{lapl1}, see \cite[Theorem 10]{cdl}. We corrected the exponent of $z$  outside the integral in the definition of $h$ there.
\begin{theorem}  
For $a>0$ and $m\ge 2$, the kernel of the radially deformed Fourier transform $\mathcal{F}_{a}$ is given by
 \begin{equation}\label{milf1}
 \begin{split}
  K_{a}^{m}(x, y)=&\, c_{a, m}\int_{0}^{1}\left[(1+2\tau)^{-\frac{\lambda}{a}}J_{\frac{2\lambda}{a}} \left(\frac{2}{a}z^{a/2}\sqrt{1+2\tau}\right)\right.\\
    &-\left.e^{-i\frac{2\pi}{a}}(1+2\tau)^{-\frac{\lambda+2}{a}}J_{\frac{2\lambda+4}{a}}\left(\frac{2}{a}z^{a/2}\sqrt{1+2\tau} \right)\right] h(z, \xi, \tau) d\tau,      
 \end{split} 
 \end{equation}
where  $
   c_{a,m}=2^{2\lambda/a}\Gamma\left(\frac{2\lambda+a}{a}\right)e^{i\frac{2\pi(\lambda+1)}{a}}\left(\frac{2}{a}\right)^{2(\lambda+2)/a},
$
and 
\begin{equation}\label{hf1}
\begin{split}
h(z, \xi, t)=&\, z^{\lambda+2}\int_{0}^{t}\zeta^{\frac{2}{a}(\lambda+1)-1}E_{\frac{2}{a}, \frac{2}{a}(\lambda+1)}^{\lambda+1}
    (b_{+}\zeta^{\frac{2}{a}})(t-\zeta)^{\frac{2}{a}(\lambda+1)-1}\\
    &\times E^{\lambda+1}_{\frac{2}{a}, \frac{2}{a}(\lambda+1)}
    (b_{-} (t-\zeta)^{\frac{2}{a}})\,{\rm d}\zeta,     
\end{split}     
\end{equation}
in which $b_{\pm}=e^{\pm i\theta}e^{i\pi/a}\left(\frac{2}{a}\right)^{2/a} z $,  
$\lambda$, $z$ and $\xi=\cos \theta$ are defined in Theorem \ref{kk1}.
\end{theorem}

We now estimate  $h(x,y, t)$ using the bound \eqref{par} for Prabhakar function. 
\begin{lemma}  Let $a>1$ and $m\ge 2$. Assume that $ \pi/a <\mu<\min\{\pi,  2\pi/a\}$ and $\mu\le |\arg e^{ i(\pm\theta+\pi/a)}|\le \pi$. Then there exists a constant $C>0$ only depending on $\mu, m$ and $a$ such that
\begin{equation} \label{phe}
    |h(z, \xi, t)|\le C z^{1/6} t^{\frac{2\lambda}{a}+\frac{1}{3a}-1}
\end{equation}
for $t\in [0, 1]$.
\end{lemma}
\begin{proof}
Using  the  inequality $x\le 1+x^{p}$ when $x\ge 0$ and $p\ge 1$,  we have  
 \begin{equation}\label{gg1}
    |b_{+}\zeta^{\frac{2}{a}}|^{\frac{\lambda}{2}+1-\frac{1}{12}}\le 1+|b_{+}\zeta^{\frac{2}{a}}|^{\lambda+1}.
 \end{equation}
By \eqref{par} and \eqref{gg1},  there exists  $C\ge 0$ such that
\begin{equation*}\begin{split}
  \left|E_{\frac{2}{a}, \frac{2}{a}(\lambda+1)}^{\lambda+1}
    (b_{+}\zeta^{\frac{2}{a}})\right|
\le \frac{C}{1+|b_{+}\zeta^{\frac{2}{a}}|^{\lambda+1}}\le \frac{C}{|b_{+}\zeta^{\frac{2}{a}}|^{\frac{\lambda}{2}+1-\frac{1}{12}}} \le \frac{C}{\left(z\zeta^{\frac{2}{a}}\right)^{\frac{\lambda}{2}+1-\frac{1}{12}}},\\
\end{split}
\end{equation*}
and similarly
\begin{equation*}
  \left|E_{\frac{2}{a}, \frac{2}{a}(\lambda+1)}^{\lambda+1} (b_{-} (t-\zeta)^{\frac{2}{a}})\right|\le \frac{C}{\left(z(t-\zeta)^{\frac{2}{a}}\right)^{\frac{\lambda}{2}+1-\frac{1}{12}}},  
\end{equation*}
when $\mu\le |\arg e^{i(\pm\theta+\pi/a)}|\le \pi$,  and $0\le \zeta\le t\le 1$. 

 It follows that
\begin{eqnarray*}
\left|h(z, \xi, t)\right|&=&z^{\lambda+2} \left|\int_{0}^{t}\zeta^{\frac{2}{a}(\lambda+1)-1}E_{\frac{2}{a}, \frac{2}{a}(\lambda+1)}^{\lambda+1}
    (b_{+}\zeta^{\frac{2}{a}})(t-\zeta)^{\frac{2}{a}(\lambda+1)-1}\right.\\
    &&\times\left. E^{\lambda+1}_{\frac{2}{a}, \frac{2}{a}(\lambda+1)}
    (b_{-} (t-\zeta)^{\frac{2}{a}})\,{\rm d}\zeta\right|\\
    &\le& C \frac{z^{\lambda+2}}{z^{\lambda+2-1/6}}\left|\int_{0}^{t}\zeta^{\frac{\lambda}{a}+\frac{1}{6a}-1}(t-\zeta)^{\frac{\lambda}{a}+\frac{1}{6a}-1}\,{\rm d}\zeta\right|\\
    &=& C \cdot z^{1/6} t^{\frac{2\lambda}{a}+\frac{1}{3a}-1}  \int_{0}^{1}u^{\frac{\lambda}{a}+\frac{1}{6a}-1}(1-u)^{\frac{\lambda}{a}+\frac{1}{6a}-1}{\rm d}u\\
    &\le& C z^{1/6} t^{\frac{2\lambda}{a}+\frac{1}{3a}-1}.
\end{eqnarray*}This completes the proof.
\end{proof}


The main estimate for $K_{a}^{m}(x, y)$ of this section is as follows,
\begin{theorem} \label{mt} Let $a>1$ and $m\ge 2$. Assume that $ \pi/a <\mu<\min\{\pi,  2\pi/a\}$,   then there exists $C>0$   only depending on $\mu, m$ and $a$ such that 
\begin{equation} \label{ed1}
    |K_{a}^{m}(x, y)|\le C, 
\end{equation}
 for  all $x, y \in \mathbb{R}^{m}$ satisfying  $\mu\le |\arg e^{ i(\pm\theta+\pi/a)}|\le \pi$.  In particular, when $a>4$, the inequality 
 \eqref{ed1} holds for any $x, y \in \mathbb{R}^{m}$ such that $\langle x, y\rangle \le 0$.
\end{theorem}
\begin{proof}  
When $|z|\ge 1$,  by \eqref{milf1} and \eqref{gg1}, we have
     \begin{eqnarray} \label{k1}
  \left|K_{a}^{m}(z, \xi)\right| &=& \left|c_{a, m}\int_{0}^{1}\left[(1+2\tau)^{-\frac{\lambda}{a}}J_{\frac{2\lambda}{a}} \left(\frac{2}{a}z^{a/2}\sqrt{1+2\tau}\right)\right.\right.\nonumber\\
    &&-\left.\left.e^{-i\frac{2\pi}{a}}(1+2\tau)^{-\frac{\lambda+2}{a}}J_{\frac{2\lambda+4}{a}}\left(\frac{2}{a}z^{a/2}\sqrt{1+2\tau} \right)\right] h(z, \xi, \tau) d\tau\right| \\     
    &\le& C \int_{0}^{1} \left(\left|J_{\frac{2\lambda}{a}} \left(\frac{2}{a}z^{a/2}\sqrt{1+2\tau}\right)\right| + \left|J_{\frac{2\lambda+4}{a}}\left(\frac{2}{a}z^{a/2}\sqrt{1+2\tau} \right)\right|\right)\nonumber\\
  && \times \left(z^{a/2}\sqrt{1+2\tau}\right)^{\frac{1}{3}} |h(z, \xi, \tau)| z^{-\frac{a}{6}} \,{\rm d}\tau \nonumber\\
    &\le& C z^{\frac{1-a}{6}}\int_{0}^{1}\tau^{\frac{2\lambda}{a}+\frac{1}{3a}-1}\,{\rm d}\tau \nonumber\\
    &\le & C.\nonumber
\end{eqnarray}
The second step is by the fact $1+2\tau\ge 1$ for $\tau\in[0, 1]$. In the third step, we have used \eqref{phe} and the following inequality  (see \cite{lan}),
\begin{equation*}
    \sup_{x\in \mathbb{R}}|x|^{1/3}|J_{\nu}(x)|\le \sup_{x\in \mathbb{R}}|x|^{1/3}|J_{0}(x)|=0.7857\ldots,  \qquad \nu>0.
\end{equation*}

When $|z|\le 1$, since the kernel series  \eqref{ks} converges absolutely and uniformly on any compact set of (see \cite[Lemma 4.17]{bko})
 \begin{equation}
     U:=\{(z, \xi)\in\mathbb{R}\times[-1, 1]\}
 \end{equation}
and it is  continuous on $U$,  there must exist a constant $C>0$ such that
\begin{equation} \label{k2}
   \left|K_{a}^{m}(z, \xi)\right|\le C.
\end{equation}
Combining \eqref{k1} and \eqref{k2}, we obtain \eqref{ed1}. 

The remaining is obtained by checking the argumental function.
\end{proof}
\begin{remark} 
Notice that 
when $a\rightarrow \infty$, we can  choose $\mu\rightarrow 0$, see  Figure \ref{fig:contour}.
This means that the domain we have not yet estimated is gradually reducing in size. 
\end{remark}

\section{Uniformly bounded kernels in dimension $2$} \label{ufb2}

The kernels of dimension two behave more nicely than those of higher dimensions. In this section, we derive uniform bounds for certain parameters $a$. In Theorem \ref{mt}, we have shown that the radially deformed Fourier kernel is bounded when $a>4$ and $\langle x, y\rangle \le 0$. Consequently, to establish a uniform bound, it suffices to consider the remaining domain, i.e. $\langle x, y\rangle >0$.

First, we give a  proof for the boundedness of $K_{8}^{2}(x, y)$. This proof will be expanded upon to accommodate other parameters.
\begin{theorem} \label{k42} When $a=8$ and $m=2$, there exists $C>0$ such that
\begin{equation*}
    \left|K_{8}^{2}(x, y)\right|\le C
\end{equation*}
for all $x, y\in \mathbb{R}^{2}$.
\end{theorem}
\begin{proof}  
We split the kernel series \eqref{ker2} into its even and odd parts,
\begin{equation*}
\begin{split}
  K_{8}^{2}(z, \xi)=&\, \underbrace{J_{0}(z_{8})+2\sum_{k=1}^{\infty}e^{-\frac{i\pi k}{4}}J_{\frac{2k}{4}}(z_{8})\cos 2k\theta}
  _{K_{4}^{2}\left(\frac{1}{2}z^{2}, \cos 2\theta\right)}\\
  &+2\sum_{k=1}^{\infty}e^{-\frac{i\pi (2k-1)}{8}}J_{\frac{2(2k-1)}{8}}(z_{8})\cos (2k-1)\theta\\
 =:&\,I+II.
\end{split}    
\end{equation*}
The even part can be expressed using the  kernel $K_{4}^{2}$ given in \eqref{k24},  which is 
 uniformly bounded by $3$ for all $x,y\in \mathbb{R}^{2}$.

By Theorem \ref{mt}, we know $|K_{8}^{2}(z, \xi)|\le C$ when $\xi\le 0$. As mentioned at the beginning of this section,  we only need to bound $K_{8}^{2}(z, -\xi)$.
Using the following property of Gegenbauer polynomials (see \cite[Eq. (4.7.4)]{sz})
\begin{equation*}
     C^{(\lambda)}_{k}(-\xi)=(-1)^{k} C^{(\lambda)}_{k}(\xi),
\end{equation*}
and the triangle inequality, we have 
\begin{equation*}
\begin{split}
    \left|K_{8}^{2}(z, -\xi)\right|=&\, |I-II|\le |I|+|II|\\
    =&\, |I|+\left|K_{8}^{2}(z, \xi)-I\right|\\ \le& \,2|I|+\left|K_{8}^{2}(z, \xi)\right|\\
    \le&\,  C,
\end{split}    
\end{equation*}
when $\xi\le 0$. This completes the proof.
\end{proof}

By induction, we record the following theorem.
\begin{theorem} \label{fdq1}
 Let $m=2$ and $a=2^{\ell}$ with $\ell\in \mathbb{N}_{0}$, there exists $C>0$ depending only on $\ell$ such that
\begin{equation*}
    \left|K_{2^{\ell}}^{2}(x, y)\right|\le C
\end{equation*}
for all $x, y\in \mathbb{R}^{2}$.
\end{theorem}

It can be further extended using the following result, see \cite[Lemma 2]{de1}.
\begin{lemma} \label{fs1} Let 
\begin{equation*}
    f(\theta)=\sum_{k=0}^{+\infty}a_{k} \cos k\theta, \qquad a_{k}\in \mathbb{C}, 
\end{equation*}
be an absolutely convergent Fourier series. Then the series 
\begin{equation*}
 g(\theta)=\sum_{k=0}^{\infty} a_{nk} \cos k\theta    
\end{equation*}
is given explicitly by 
\begin{equation*}
    g(\theta)=\frac{1}{n}\sum_{j=0}^{n-1}f\left(\frac{\theta+2\pi j}{n}\right).
\end{equation*}  
\end{lemma}

This  leads to the following, 
\begin{theorem} \label{kj1}
 Let $m=2$ and $a=2^{\ell}/n$ with $\ell \in \mathbb{N}_{0}$ and $n\in \mathbb{N}$, there exists $C>0$ depending only on $a$ such that
\begin{equation*}
    \left|K_{\frac{2^{\ell}}{n}}^{2}(x, y)\right|\le C
\end{equation*}
for all $x, y\in \mathbb{R}^{2}$.
\end{theorem}
\begin{proof} By Lemma \ref{fs1}, we have
\begin{equation*}
   K_{\frac{2^{\ell}}{n}}^{2}(z, \cos \theta)= \frac{1}{n}\sum_{j=0}^{n-1} K_{2^{\ell}}^{2}\left( n z^{1/n}, \cos \left(\frac{\theta+2\pi j}{n}\right)\right). 
\end{equation*}
  Then the bound follows from Theorem \ref{fdq1}.  
\end{proof}

Using the above uniform bound and the Plancherel theorem, now the $L^{p}$-boundedness of $\mathcal{F}_{2^{\ell}/n}$ follows by the Riesz-Thorin interpolation theorem. 
\begin{theorem}[Hausdorff-Young inequality] Let $m=2$ and $a=2^{\ell}/n$ with $\ell\in \mathbb{N}_{0}$ and $n\in \mathbb{N}$. For $1\le p\le 2$ and $1/p+1/p'=1$, there exists a constant $c(p, a)>0$   such that
\begin{equation}
  \left \|\mathcal{F}_{\frac{2^{\ell}}{n}}(f)\right\|_{L^{p'}(\mathbb{R}^{2}, \,{\rm d}\mu(x))} \le c(p, a) \|f\|_{L^{p}(\mathbb{R}^{2}, \,{\rm d}\mu(x))},
\end{equation}
 for all $f\in L^{p}\left(\mathbb{R}^{2}, \,{\rm d}\mu(x)\right)$ with $\,{\rm d}\mu(x)=|x|^{\frac{2^{\ell}}{n}-2}{\rm d}x$.  
\end{theorem}
\begin{remark}
    Using these bounds,  we can now extend the applicability of  the results in  \cite{tjo} and \cite{kr}. It is also interesting to extend the results for translation operator and multiplier theorems  (see e.g. \cite{bede, dh})  for the generalized Fourier transforms with parameters $a$ considered in this work.
\end{remark}


\section*{Acknowledgements}
 The second author was supported  by NSFC Grant No.12101451 and  China Scholarship Council.
\bibliographystyle{amsplain}

\begin{thebibliography}{99}
\bibitem{bede} S. Ben Sa\"id, L. Deleaval, Translation operator and maximal function for the $(\kappa, 1)$-generalized Fourier transform, {\em J. Funct. Anal.} {\bf 279} (2020), 108706.
\bibitem{bko1} S. Ben Sa\"id, T. Kobayashi, B. {\O}rsted, Generalized Fourier transforms $\mathcal{F}_{\kappa,a}$, {\em C. R.
Acad. Sci. Paris Ser. I.} {\bf 347} (2009),  1119--1124.
\bibitem{bko} S. Ben Sa\"id, T. Kobayashi, B. {\O}rsted, Laguerre semigroup and Dunkl operators, {\em Compos. Math.} {\bf 148} (2009), 1265--1336.
\bibitem{bry} Y.A. Brychkov, {\em Handbook of special functions. Derivatives, integrals,
series and other formulas.} Chapman and Hall/ CRC, Boca Raton, 2009.
\bibitem{cdl} D. Constales, H. De Bie,  P. Lian, Explicit formulas for the Dunkl dihedral kernel and the $(\kappa, a)$-generalized Fourier kernel, {\em J. Math. Anal. Appl.} {\bf 460} (2018), 900--926.
\bibitem{de1} H. De Bie, The kernel of the radially deformed Fourier transform, {\em Integral Transforms Spec. Funct.} {\bf 24} (2013), 1000--1008.
\bibitem{dov} H. De Bie, R. Oste, J. Van der Jeugt, Generalized Fourier transforms arising from the enveloping algebras of $\mathfrak{sl}(2)$ and $\mathfrak{osp}(1|2)$, {\em Int. Math. Res. Not. IMRN.} {\bf 15} (2016), 4649--4705.
\bibitem{dx} H. De Bie, Y. Xu, On the Clifford-Fourier transform, {\em Int. Math. Res. Not. IMRN.} {\bf 22} (2011), 5123--5163.
\bibitem{dd} L. Deleaval, N. Demni, On a Neumann-type series for modified Bessel functions of the kind, {\em Proc. Amer. Math. Soc.} {\bf 146} (2018), 2149--2161.
\bibitem{dej} M.F.E. de Jeu, The Dunkl transform,  {\em Invent. Math.}  {\bf 113} (1993), 147--162.
\bibitem{dux} C.F. Dunkl, Y. Xu, {\em Orthogonal polynomials of several variables,} Cambridge university press, 2014.
\bibitem{dh} J. Dziuba\'nski, A. Hejna, H\"ormander's multiplier theorem for the Dunkl transform, {\em J. Funct. Anal.} {\bf 277}, (2019),  2133-2159.  
\bibitem{garg} R. Garra, R. Garrappa, The Prabhakar or three parameter Mittag–Leffler function: Theory and application, {\em Commun. Nonlinear Sci. Numer. Simul.} {\bf 56} (2018), 314--329.
\bibitem{GJ} S. Ghobber, P. Jaming, Uncertainty principles for integral operators, {\em Studia Math.}
{\bf 220} (2014), 197--220.
\bibitem{git} D.V. Gorbachev, V.I. Ivanov, S.Y. Tikhonov, Pitt’s inequalities and uncertainty principle for generalized Fourier transform, {\em Int. Math. Res. Not. IMRN.} {\bf 23} (2016), 7179--7200.
\bibitem{gitt1} D.V. Gorbachev, V.I. Ivanov, S.Y. Tikhonov, On the kernel of the $(\kappa, a)$-generalized Fourier transform, {\em Forum Math. Sigma.} {\bf 11} (2023), doi: 10.1017/fms.2023.69.
\bibitem{Gorenflo2014} R. Gorenflo, A.A. Kilbas, F. Mainardi, S.V. Rogosin, {\em Mittag-Leffler functions, related topics and application,} Springer, Berlin, Heidelberg, 2014.
\bibitem{how} R. Howe, The oscillator semigroup, The mathematical heritage of Hermann Weyl (Durham, NC,
1987), Proc. Sympos. Pure Math., vol. 48, Amer. Math. Soc., Providence, RI, 1988, pp. 61–132.
\bibitem{tjo} T.R. Johansen, Weighted inequalities and uncertainty principles for the $(\kappa, a)$-generalized Fourier transform, {\em Int. J. Math.} {\bf 27} (2016), 1650019.
\bibitem{km} T. Kobayashi, G. Mano, The Schr\"odinger model for the minimal representation of the indefinite
orthogonal group $O(p, q)$, {\em Mem. Amer. Math. Soc.} {\bf 213} (2011), no. 1000, vi+132.
\bibitem{kr} V. Kumar, M. Ruzhansky, $L^{p}-L^{q}$ Boundedness of $(\kappa, a)$-Fourier multipliers with applications to nonlinear equations, {\em Int. Math. Res. Not. IMRN.} {\bf 2} (2023),  1073--1093.
\bibitem{lan} L.J. Landau. Bessel functions: monotonicity and bounds. {\em J. Lond. Math. Soc.}  {\bf 61} (2000), 197--215.
\bibitem{fmkvb}F. Maes, K. Van Bockstal, Existence and uniqueness of a weak solution to fractional single-phase-lag heat equation, {\em Fract. Calc. Appl. Anal.} {\bf 26} (2023), 1663--1690
\bibitem{mos} W. Magnus, F. Oberhettinger, R.P. Soni, {\em Formulas and theorems for
the special functions of mathematical physics,} third enlarged edition, Springer,
New York, 1966.
\bibitem{olb} F.W.J. Olver, D.W. Lozier, R.F. Boisvert, and C.W. Clark, {\em NIST handbook of mathematical
functions.} Cambridge University Press, 1 pap/cdr edition, 2010.
\bibitem{pod} I. Podlubny, {\em Fractional differential equations,} Academic press, San Diego, 1999.
\bibitem{pbm2} A.P. Prudnikov, Y.A. Brychkov, O.I. Marichev, {\em Integrals and series, vol. 2 of Special functions,} New York: Gordonand Breach, 1992.
\bibitem{pbm} A.P. Prudnikov, Y.A. Brychkov, O.I. Marichev, {\em Integrals and series, vol. 5 of Inverse Laplace transforms,} New York: Gordonand Breach, 1992.
\bibitem{sz} G. Szeg\"o, {\em Orthogonal polynomials,} Amer. Math. Soc. Colloq. Publ. Vol. 23, Providence, 4th
edition, 1975.


\end{thebibliography}

\end{document}